\documentclass[11pt]{article}
\usepackage{amsmath,amssymb,amsfonts,amsthm}
\usepackage{hyperref} 
\usepackage{graphicx} 
\usepackage{enumitem} 

\usepackage{graphicx}%
\usepackage{multirow}%
\usepackage{amsmath,amssymb,amsfonts}%
\usepackage{amsthm}%
\usepackage{mathrsfs}%
\usepackage[title]{appendix}%
\usepackage{xcolor}%
\usepackage{textcomp}%
\usepackage{manyfoot}%
\usepackage{booktabs}%
\usepackage{algorithm}%
\usepackage{algorithmicx}%
\usepackage{algpseudocode}%
\usepackage{listings}%
\usepackage{mathrsfs}
\usepackage{subcaption}
\usepackage{booktabs}
\usepackage{mathrsfs}
\usepackage{multirow}
\usepackage{float}
\usepackage{amssymb}
\usepackage{cancel}

\newtheorem{theorem}{Theorem}[section]
\newtheorem{lem}[theorem]{Lemma}
\newtheorem{proposition}[theorem]{Proposition}
\newtheorem{cor}[theorem]{Corollary}
\newtheorem{definition}{Definition}

\newtheorem{rem}{Remark}

\title{On point and block primitive designs invariant under permutation groups}
\author{Amin Saeidi\\
\small Department of Mathematics and applied mathematics\\
\small University of Limpopo, South Africa\\
\small \texttt{amin.saeidi@ul.ac.za}
}
\date{}

\begin{document}

\maketitle

\begin{abstract}
In this paper, we present a method for constructing point primitive block transitive $t$-designs invariant under finite groups. Furthermore, we demonstrate that every point and block primitive $G$-invariant design can be generated using this method.

Additionally, we establish the theoretical possibility of identifying all block transitive $G$-invariant designs. However, in practice, the feasibility of enumerating all designs for larger groups may be limited by the computational complexity involved.
\end{abstract}

\section{Introduction}
A very common technique in the construction of $1$-designs invariant under finite groups is using transitive group actions. For instance, in~\cite{nato}, two methods for constructing designs from finite primitive groups have been discussed. The first method involves symmetric designs constructed through primitive permutation representations, while the second method utilizes $G$-conjugacy classes of $G$. Both methods have been widely applied in recent years to construct $1$-designs invariant under finite simple groups, including sporadic simple groups~\cite{{dik},{mr2},{mr4},{ms}} and families of finite simple groups~\cite{{xr2},{mrsz1},{mrsz2},{ms2},{ms1}}. Additionally, Moori introduced a third method to construct designs from the fixed points of a permutation group (see~\cite{m3} and also~\cite{an, s}).

While Key-Moori methods provide valuable tools for creating designs from finite simple groups, the designs constructed by these methods are mostly $1$-designs. To construct $t$-designs for $t \ge 2$, one can consider multiply transitive actions. A well-known technique is due to Kramer and Mesner~\cite{kmo}, where a group of permutations is assumed to be an automorphism group of the design. The goal is to choose a suitable subset of the $k$-element orbits, covering every $t$-element orbit exactly $\lambda$ times. This method has been used for the construction of $t$-designs, as seen in works such as \cite{blw, lw}. Additionally, in \cite{cmot}, the authors investigated $3$-designs from the families of $PSL(2,q)$ and constructed $3$-designs with some restrictions on the sizes of blocks.

This paper aims to present a new method for constructing $G$-invariant point and block primitive designs from finite permutation groups. We demonstrate that this method generalizes Key-Moori Method 1, and we highlight its applications with some examples. Furthermore, we establish that the converse of the method holds; meaning that every $G$-invariant point and block primitive design can be constructed using our method by choosing suitable maximal subgroups.
To illustrate our method, we present the complete list of point and block primitive 2-designs invariant under the Higman-Sims group $HS$ and the Conway group $Co_3$, with respect to their 2-transitive actions. This serves as the main result of this paper.\\

\begin{theorem} \label{method}
Let $G|\Omega$ be a primitive group action of degree $n$, and assume that $M$ is a maximal subgroup of $G$ with $|M| \le |G_\alpha|$. Let $\Delta = \alpha^M$ be an orbit of the action of $M$ on $\Omega$. Then, we can construct a $G$-invariant point and block primitive $1$-design with parameters $(n, |\Delta|, \frac{|G_\alpha| \times |\Delta|}{|M|})$, where the block set is $\{\Delta^g \mid g \in G\}$. Conversely, every $G$-invariant point and block primitive $t$-design can be constructed using this method by possibly by "merging" the blocks as necessary.
\end{theorem}

\begin{rem}
If $G$ is $t$-transitive for $t \ge 2$, then the design has a $t$-design structure.
\end{rem}

\begin{rem}

The method can be considered as a generalization of Key-Moori Method 1. In fact, if $G|\Omega$ is primitive and $M$ is a point stabilizer, then $n = |G{:}M|$ and we have a $(n, k, k)$. Also if $t \ge 2$, the design constructed by Key-Moori Method 1 will be trivial while with the generalization, we may actually construct $t$-designs.

\end{rem}

We also aim to provide a description of all point and block transitive designs. We demonstrate that if $G$ is a group acting on a set of size $n$ (for $2 \le n$), the power set $\mathcal{P}(n)$ has a unique partition that can be considered as bases for blocks of a block transitive design. Leveraging this observation, we have computed the set of all block transitive designs invariant under the Mathieu group $M_{22}$. However, in practice, finding all block transitive designs may be challenging due to the computational complexity, especially for larger groups.

  \section{Terminology and notation}

  Our notation for designs follows~\cite{ak}. Let $\mathcal{D} = (\mathcal{P}, \mathcal{B}, \mathcal{I})$ be an incidence structure, defined by a triple with a point set $\mathcal{P}$, a block set $\mathcal{B}$ (disjoint from $\mathcal{P}$), and an incidence set $\mathcal{I} \subseteq \mathcal{P} \times \mathcal{B}$. If the ordered pair $(p, B) \in \mathcal{I}$, we say that $p$ is incident with $B$. It is often convenient to assume that the blocks in $\mathcal{B}$ are subsets of $\mathcal{P}$, so that $(p, B) \in \mathcal{I}$ if and only if $p \in B$.

For a positive integer $t$, we define $\mathcal{D}$ as a $t$-design if every block $B \in \mathcal{B}$ is incident with exactly $k$ points, and every $t$ distinct points are together incident with $\lambda_t$ blocks. In this case, we say $\mathcal{D}$ is a $t$-$(v, k, \lambda_t)$ design, where $v = |\mathcal{P}|$.

    If $\mathcal{D}$ is a $t$-design, then it is also an $s$-$(v, k, \lambda_s)$ design for $1\le s \le t$. We can use the following equation to find $\lambda_s$.

\begin{equation} \label{lamda}
\lambda_s = \lambda_t\frac{\binom{v - s}{t - s}}{   \binom{k - s}{t - s}}.
\end{equation}

The full parameters of a design are typically given by $(t, b, v, k, r, \lambda_t)$, where $b$ is the number of blocks, and $r = \lambda_1$ is the number of blocks incident with a given point. A symmetric design is a $t$-design with the same number of points and blocks, i.e., $v = b$.

An automorphism of a design $\mathcal{D}$ is a permutation of the points and blocks that preserves the incidence relationship. This means that if a point is incident with a block, the images of the point and block under the permutation are also incident. The set of all automorphisms of $\mathcal{D}$ (which is obviously a group) is called the automorphism group of $\mathcal{D}$, denoted by $Aut(\mathcal{D})$.


A $t$-$(v, k, 1)$ design is called a Steiner system (usually for $t \ge 2$) and is denoted by $S(t, k, v)$. A Steiner system is termed non-trivial if $t < k < v$. The proof of existence and uniqueness of Steiner systems is typically challenging, especially for large values of $t$. For instance, there exists only one Steiner system with parameters $S(3, 4, 10)$ (up to isomorphism), while there are 1054163 non-isomorphic Steiner systems with parameters $S(3, 4, 16)$ (see, for example,~\cite{pok} for a reference).

By $G|\Omega$, we mean the group action of $G$ on the set $\Omega$. For any subset $B$ of $\Omega$, the stabilizer of $B$ in $G$ is denoted as $G_B$, given by $G_B = \{g \in G \mid B^g = B\}$. If $B = \{\alpha\}$, the stabilizer $G_B$ corresponds to the stabilizer of the point $\alpha$, denoted as $G_\alpha$.

A transitive action is considered primitive if $G_{\alpha}$ is a maximal subgroup of $G$. If $G$ is a transitive group of degree $n$, then, by Cayley's Theorem, we can assume that $G$ acts on the set $\underline{n} = \{1, 2, ..., n\}$, and $G \le S_n$. So, when we say $G \le S_n$, we mean that we are considering $G$ as a group acting on the set $\underline{n}$.

If $G$ is a subgroup of $Aut(\mathcal{D})$, then we say that $\mathcal{D}$ is $G$-invariant. The design is point-transitive (respectively block-transitive) if $G$ acts transitively on the points (respectively blocks) of $\mathcal{D}$. Point-primitive and block-primitive designs can be defined analogously.


\section{Block transitive designs from the group actions}

\begin{definition}
Let $G|\Omega$ be a transitive group action, and let $\mathcal{P}_k(\Omega)$ be the set of all subsets of $\Omega$ of order $k$.
Choose an arbitrary subset $\Delta_1 \in \mathcal{P}_k(\Omega)$ and define the set $\mathcal{B}_1 = \{\Delta_1^g \mid g \in G\}$. If $\mathcal{B}_1 \neq \mathcal{P}_k(\Omega)$, choose $\Delta_2 \in \mathcal{P}_k(\Omega) - \mathcal{B}_1$ and define $\mathcal{B}_2 = \{\Delta_2^g \mid g \in G\}$. Continue this process until all subsets of size $k$ are covered. The set of all $\mathcal{B}_i$ is denoted by $\Sigma_k(G|\Omega)$ (or $\Sigma_k(G|n)$ where $n = |\Omega|$).
\end{definition}

\begin{lem}\label{set}
Let $G|\Omega$ be a transitive group action. Then $\Sigma_k(G|\Omega)$ is a partition of $\mathcal{P}_k(\Omega)$.
\end{lem}

\begin{proof}
By construction, each set $\mathcal{B}_i$ is defined to be non-empty. Also, the union of all sets in $\Sigma_k(G|\Omega)$ is equal to $\mathcal{P}_k(\Omega)$. It remains to show that no two distinct sets in $\Sigma_k(G|\Omega)$ have a common element.

Let $\mathcal{B}_i$ and $\mathcal{B}_j$ be two distinct sets in $\Sigma_k(G|\Omega)$. Suppose, for contradiction, there exists a set $C$ that belongs to both sets. Then $B_i ^g = C = B_j ^h$ for some $g, h \in G$. Therefore, $B_j = B_i ^{gh^{-1}m} \in \mathcal{B}_i$ and the result follows.
\end{proof}


\begin{theorem} \label{A}
Let $G|\Omega$ be a transitive group action, and let $\Delta \in \mathcal{P}_k(\Omega)$. Then for every $b \in \Sigma_k(G|\Omega)$, we can construct a $1$-$(|\Omega|, k, \frac{kb}{n})$ block transitive design with respect to $G$, denoted by $\mathcal{D}_k(G|\Omega)$. Moreover, the number of blocks $b$ is given by $|G:G_\Delta|$. Conversely, every block transitive $G$-invariant $1$-$(n,k,\lambda)$ design is isomorphic to some $\mathcal{D}_k(G|\Omega)$.
\end{theorem}

\begin{proof}
Assume that $\mathcal{B} = \{\Delta^g | g \in G\}$. We will show that $\mathcal{B}$ constitutes a block set for $\mathcal{D}_k(G|\Omega)$.
Assume that $\Delta_1, \ldots, \Delta_\lambda$ are distinct elements of $\mathcal{B}$ such that each $\Delta_i$ contains the point $x \in \Omega$, and no other element in $\mathcal{B}$ contains $x$. For an arbitrary element $y \in \Omega$, let $y = x^g$. Then, it is clear that $\Delta_1^g, \ldots, \Delta_\lambda^g$ are distinct elements of $\mathcal{B}$ such that each $\Delta_i^g$ contains the point $y$, and no other element in $\mathcal{B}$ contains $y$.

To find the number of blocks, we define a map $G/G_\Delta \to \mathcal{B}$, with $xG_\Delta \mapsto \Delta$. Notice that $G / G_\Delta$ is the set of all left cosets of $G_\Delta$ in $G$. It is easy to see that the map is a bijection.

Conversely, assume that $\mathcal{D}$ is a block transitive $G$-invariant design. Let $\Omega$ be the point set of $\mathcal{D}$, and let $\Delta$ be any block of $\mathcal{D}$. We define $\mathcal{B} = \{\Delta^g | g \in G\}$. As the design is block transitive, we conclude that $\mathcal{B}$ is equal to the block set of $\mathcal{D}$, and the result follows.
\end{proof}

\begin{rem}
If $\Sigma_k(G|\Omega) = \{\binom{n}{k}\}$, then we say that the partition is trivial. In this case, each subset of size $k$ is a block, so the automorphism group of the design will be $S_n$. This also implies that the trivial design with all possible blocks is the only block transitive design over $G$.
\end{rem}


\begin{rem}
According to the results of this section, we can theoretically find all block transitive $G$-invariant designs for a given action of a group $G$. To achieve this, we need to determine the partition $\Sigma_k(G|\Omega)$ for each $k$. However, due to the complexity of the computations, it may not be feasible to find this partition for all values of $k$, especially for larger values of $n$. For instance, we can easily verify using MAGMA that if $G$ is the Mathieu group $M_{11}$, then we have $\Sigma_5(G|11) = \Sigma_6(G|\Omega) = \{66, 396\}$. Similarly, for $G = M_{12}$, we find $\Sigma_6(G|12) = \{132, 792\}$. For other values of $k$, the partitions are trivial. As a result, we can determine the parameters of all block transitive designs under these groups and their actions of degree 11 and 12, respectively.
\end{rem}

 \begin{rem} \label{sig}
 Another useful observation is that  $\Sigma_k(G|n ) =  \Sigma_{n-k}(G|n )$. Therefore, each design constructed by a block in $\Sigma_k(G|n)$ corresponds to a design constructed by a block in  $\Sigma_{n-k}(G|n)$. These corresponding designs are, in fact, complements. In other words, if $B$ is a block of one design, then $\Omega - B$ is a block of the complement design.

\end{rem}

\begin{proposition} \label{sigma}
Let $G|\Omega$ be the Mathieu group $M_{22}$ of degree $22$. Then $\Sigma_k(G|22 )$ is trivial for $k \in \{1,2,3\}$.  Moreover,
  \begin{itemize}
   \fontsize{8pt}{9pt}\selectfont
      \item  $\Sigma_4(G|22 ) =  \Sigma_{18}(G|22 ) = \{ 1155, 6160 \}$.
    \item  $\Sigma_5(G|22 ) = \Sigma_{17}(G|22 )=   \{ 462, 3696^2, 18480 \}$.
    \item  $\Sigma_6(G|22 ) = \Sigma_{16}(G|22 ) =   \{ 77, 1232^2, 7392, 9240, 55440  \}$.
    \item  $\Sigma_7(G|22 ) = \Sigma_{15}(G|22 ) =   \{ 176^2, 1232, 2640, 18480^2, 55440, 73920 \}$.
    \item  $\Sigma_8(G|22 ) =  \Sigma_{14}(G|22 ) =  \{  330, 2640^2, 9240, 27720, 36960^2, 55440, 73920^2 \}$.

       \item  $\Sigma_9(G|22 ) =  \Sigma_{13}(G|22 ) =  \{  4620, 6160, 18480^3, 24640, 36960^2, 110880, 221760 \}$.

          \item  $\Sigma_{10}(G|22 ) = \Sigma_{12}(G|22 ) =    \{   616, 2310, 6160, 7392^2, 18480, 22176, 27720, 73920^3, 110880, 221760 \}$.

  \item  $\Sigma_{11}(G|22 ) =   \{   672^2, 7392^2, 9240^2, 27720, 36960^2, 44352^2, 73920^2, 110880^3\}$.

  \end{itemize}

\end{proposition}


\begin{proof}
  The computations are based on MAGMA.
\end{proof}

\begin{rem}
As Proposition~\ref{sigma} suggests, every $M_{22}$ block transitive $t$-design with parameters $(22,k, \lambda)$ has exactly $b$ blocks for some $b \in \Sigma_{k}(G|22)$. In some cases, the partition contains multiple subsets of a given size, which may result in non-isomorphic designs with the same parameters. For example, $\Sigma_{10}(G|22)$ contains exactly three sets of size 73920, constructing three $3$-$(22, 10, 5760)$ designs. However, only two of the three are isomorphic.
\end{rem}

\begin{cor}
Table 1 lists of all block transitive designs over $M_{22}$ and its action of degree 22 for $4 \le k \le 11$.
\end{cor}


\begin{table}[h]
 \fontsize{8pt}{9pt}\selectfont
  \begin{subtable}{0.45\textwidth}
    \centering
    \captionsetup{justification=centering}

    \begin{tabular}{|c|c|c|c|c|c|c|}
      \hline
     $t$ & $v$ & $b$ & $r$ & $k$ & $\lambda$  \\ \hline

 3   &  22  &   1155 &  210 &   4  &       3      \\ \hline
 3   &   22  &  6160  & 1120 &   4   &      16     \\ \hline

 3   &   22  &   462 &105 &5   &        3    \\ \hline
 3   &   22  &   36960 & 8400&5   &        240     \\ \hline
 3   &   22  &   18480 & 4200 &5   &        120    \\ \hline

 3   &   22  &   77 & 21 & 6   &       1    \\ \hline
 3   &   22  &   1232&  336& 6   &        16  \\ \hline
 3   &   22  &   7392 & 2016 & 6   &        96     \\ \hline
3   &   22  &   9240  & 2520& 6   &        120   \\ \hline
 3   &   22  &   55440  &15120& 6   &        720    \\ \hline

 3   &   22  &   176  & 56 &7   &        4    \\ \hline
 3   &   22  &   1232  &392 &7   &       28   \\ \hline
 3   &   22  &   2640  & 840&7   &        60   \\ \hline
 3   &   22  &   18480  &5880 &7   &        420    \\ \hline
 3   &   22  &   55440  &17640 &7   &       1260  \\ \hline
 3   &   22  &   73920  &23520 &7   &       1680  \\  \hline

3   &   22  &   330  & 120&8   &        12     \\ \hline
 3   &   22  &  2640  &960 &8  &        96     \\ \hline
 3   &   22  &   9240 & 3360& 8   &        336   \\ \hline
3   &   22  &   27720 &10080 &8   &       1008\\ \hline
 3   &   22  &  36960  &13440& 8   &       1344   \\ \hline
 3   &   22  &   55440 &20160 &8   &       2016   \\ \hline
 3   &   22 &  73920  & 26880&8   &        2688    \\ \hline

 3   &  22  &  4620   &1890 &  9  &       252      \\ \hline

    \end{tabular}
  \end{subtable}%
  \hfill
  \begin{subtable}{0.5\textwidth}
  \hfill
    \centering
    \captionsetup{justification=centering}
    \begin{tabular}{|c|c|c|c|c|c|c|}
      \hline
    $t$ & $v$ & $b$ & $r$ & $k$ & $\lambda$  \\ \hline

 3   &   22  &   6160  & 2520&9  &        336    \\ \hline
 3   &  22   &  18480   & 7560&  9  &       1008   \\ \hline
 3   &   22  &   24640  &10080 &9  &        1344    \\ \hline
 3   &   22  &   36960  &15120 &9   &        2016     \\ \hline
 3   &   22  &   110880  & 45360 &9   &       6048 \\ \hline
 3   &   22  &   221760  & 90720&9   &         12096   \\ \hline

 3   &   22  &   616  & 280 &10   &       48     \\ \hline
 3   &   22  &   2310  & 1050&10   &     180    \\ \hline
 3   &   22  &   6160  & 2800&10   &        480     \\ \hline
 3   &   22  &   7392  & 3360&10   &        576    \\ \hline
 3   &   22  &   18480  & 8400 & 10   &      1440    \\ \hline
 3   &   22  &   22176  &10080 & 10   &        1728     \\ \hline
 3   &   22  &   27720  & 12600 &10   &        2160    \\ \hline
 3   &   22  &  73920  & 33600&10   &      5760   \\ \hline
 3   &   22  &   110880  &50400& 10   &       8640   \\ \hline
 3   &   22  &   221760  &100800 &10   &       17280   \\ \hline

 3   &   22  &  672  & 336&11  &        72     \\ \hline
 3   &   22  &   7392  &3696& 11   &        792     \\ \hline
 3   &   22  &   9240  & 4620&11   &        990   \\ \hline
 3   &   22  &  27720  & 13860&11   &        2970 \\ \hline
 3   &   22  &   36960  &18480 &11   &        3960   \\ \hline
 3   &   22  &   44352  & 22176 &11   &       4752   \\ \hline
 3   &   22  & 73920 & 36960 &11  &        7920     \\ \hline
 3   &   22  &    110880 & 55440& 11   &      11880   \\ \hline

    \end{tabular}
  \end{subtable}

  \caption{ $G$-invariant block transitive 3-designs for $G = M_{22}$.}
  \label{fig-2}

\end{table}

\begin{rem}
For $k>11$, every block transitive $G$-invariant design is a complement of one of the designs in the table. Also, the automorphism group of each design in Table 1 is either $M_{22}$ or $M_{22}{:}2$. In fact it is $M_{22}$ if and only if there is another design isomorphic to the design in the table.

\end{rem}

\newpage

\section{ On point and block primitive designs}
In this section, we prove Theorem~\ref{method}. To illustrate our method, we consider the 2-transitive actions of the Higman-Sims group $HS$ and the Convey group $Co_3$. As it is well-known, both $HS$ and $Co_3$ are 2-transitive groups with degrees 176 and 276, respectively.

\begin{lem} \label{ifpart}
Let $G|\Omega$ be a primitive group action of degree $n$. Assume that $M$ is a maximal subgroup of $G$ with $|M| \leq |G_\alpha|$. Let $\Delta = \alpha^M$ be an orbit of the action of $M$ on $\Omega$. Then, we can construct a $1$-$(n, |\Delta|, \frac{|G_\alpha| \times |\Delta|}{|M|})$ design, where the block set is $\{\Delta^g \mid g \in G\}$. Moreover, $G$ acts primitively on both points and blocks of the design.
\end{lem}

\begin{proof}
By Lemma~\ref{A}, we can construct a $1$-$(n, |\Delta|, \frac{|\Delta|b}{n})$ design, where $b = |G{:}G_\Delta|$. Also, note that $n = |\Omega| = |G{:}G_\alpha|$. If $x \in M$, then $\Delta^x = (\alpha^M)^x = \alpha^M = \Delta$. Therefore, $M \leq G_\Delta$. The maximality of $M$ implies that $G_\Delta = M$ or $G$.

If $G_\Delta = G$, then $b = 1$, $\Delta = \Omega$, and we have a trivial $1$-$(n, n, 1)$ design. Now, assume that $G_\Delta = M$, then we have:
$$\frac{|\Delta|b}{n} = \frac{|\Delta| \cdot |G| \cdot |G_\alpha|}{|G| \cdot |G_\Delta|} = \frac{|G_\alpha| \times |\Delta|}{|M|}.$$
The primitivity of the action of $G$ on the set of points and blocks is obvious since both stabilizers are maximal in $G$.
\end{proof}

\begin{rem} \label{fine}
Let $G|\Omega$ be a primitive group action of degree $n$, and assume that $\Delta_1 = {\alpha_1}^M$ is an orbit of the action of $M$ on $\Omega$. Now choose $\alpha_2 \in \Omega - \alpha_1^M$ and set $\Delta_2 = {\alpha_2}^M$. Continue this process until no other points in $\Omega$ are left. Then we have:
$$\Omega = \Delta_1 \cup \Delta_2 \cup \ldots \cup \Delta_s.$$

It is clear that $\Delta_i$ are $G$-invariant subsets of $\Omega$, all having the same size. Also, if $\Delta$ is a union of $r$ $\Delta_i$'s, then we can apply Lemma~\ref{ifpart} to construct a block transitive design. In this case, the design will be a $1$-$(n, k, \frac{|G_\alpha| \times k}{|M|})$ design, where $k = \sum\limits_{i = 1}^s {{\Delta _i}}$. We call this new design a merging of designs constructed by Lemma~\ref{ifpart}.
\end{rem}

\begin{lem} \label{onlyif}
Let $\mathcal{D}$ a point and block primitive design. Then $\mathcal{D}$ is a merging of some designs constructed by Lemma~\ref{ifpart}.

\end{lem}
\begin{proof}
Let $B$ be a block of the design. Since the design is block primitive, we have $M = G_B$ is a maximal subgroup of $G$. We choose $\alpha_1 \in B$ and note that $\alpha_1^M \subseteq B^M = B$. Now assume that $\alpha_2 \in B - \alpha_1^M$. Let $\Delta = \alpha_2^M$. If $\gamma \in \alpha^M \cap \beta^M$, then $\alpha \in \beta^M$, which is a contradiction. Therefore, $\alpha_1^M \cap \alpha_2^M = \emptyset$ and $B$ contains the union of $\alpha_1^M$ and $\alpha_2^M$. Continuing this process, we can see that $B$ is a union of $\alpha_i^M$ and the result follows.
\end{proof}

We conclude this section by applying the method to the Higman-Sims group $HS$ and the Conway group $Co_3$. The computations are based on MAGMA. In Table 3, all designs are $3$-designs with $v = 276$. To accommodate the table within the available space, we have removed the columns $t$ and $v$. The automorphism group of each design in Table 2 and Table 3 is $HS$ and $Co_3$, respectively.


\begin{table}[htbp]
 \fontsize{7pt}{8pt}\selectfont
  \begin{subtable}{0.50\textwidth}
    \centering
        \begin{tabular}{|c|c|c|c|c|c|c|}
      \hline

      Max & $t$ & $v$ & $b$ & $r$ & $k$ & $\lambda$ \\ \hline

$M_2$ & 2   &   176  & 176   &  50   &   50  &       14      \\ \hline
$M_2$ & 2   &   176  &  176  &   126  & 126   &        90     \\ \hline

$M_2$ & 2   &   176  & 1100    & 350    &   56  &        110     \\ \hline
$M_2$ & 2   &   176  &   1100  &    750 & 120   &           510  \\ \hline

$M_3$ & 2   &   176  & 1100    & 50    &   8 &       2      \\ \hline
$M_3$ & 2   &   176  &   1100  &   1050  & 168    &   1002          \\ \hline

$M_4$ & 2   &   176  & 3850    &   1750  &   80  &        790     \\ \hline
$M_4$ & 2   &   176  & 3850    &   2100  &   96  &           1140  \\ \hline

$M_5$ & 2   &   176  & 4125    &   1500  &   64 &          540   \\ \hline
$M_5$ & 2   &   176  & 4125    &    2625 &   112 &            1665 \\ \hline

$M_6$ & 2   &   176  & 5775    &   525  &  16 &      45       \\ \hline

    \end{tabular}
  \end{subtable}%
  \hfill
   \begin{subtable}{0.5\textwidth}
       \centering

    \begin{tabular}{|c|c|c|c|c|c|c|}
      \hline
      Max & $t$ & $v$ & $b$ & $r$ & $k$ & $\lambda$ \\ \hline
$M_6$ & 2   &   176  & 5775    &  5250   &  160&     4770        \\ \hline

$M_7$ & 2   &   176  & 5600    &   2100  &  66 &         780    \\ \hline
$M_7$ & 2   &   176  & 5600    &   3500  &  110 &           2180  \\ \hline

$M_7$ & 2   &   176  & 5600    &  350   &  11 &  20           \\ \hline
$M_7$ & 2   &   176  & 5600    &    5250 &  165 &  4920           \\ \hline

$M_8$ & 2   &   176  &  15400    &    1050 &  12 &  66  \\ \hline
$M_8$ & 2   &   176  &  15400    &   6300  &  72 &      2556      \\ \hline
$M_8$ & 2   &   176  &  15400    &  7875   &  90 &  4005         \\ \hline

 $M_9$ & 2   &   176  &  36960    &   1260  &  6 &     36    \\ \hline
  $M_9$ & 2   &   176  &  36960    &   10500  &  50 &      2940       \\ \hline
     $M_9$ & 2   &   176  &  36960    &    25200 &  120 &      17136       \\ \hline

    \end{tabular}
  \end{subtable}

  \caption{Some $G$-invariant block transitive 2-designs for $G = HS$.}
  \label{fig-2}

\end{table}


\begin{table}[h]
\begin{subtable}{0.50\textwidth}
\centering
\fontsize{6pt}{8pt}\selectfont
\begin{tabular}[H]{|c|c|c|c|c|c|} \hline
  Max & $b$ & $r$ & $k$ & $\lambda$ \\ \hline \hline
  $M_1$ & 344282400 & 44906400& 36  & 5715360         \\ \hline
  $M_1$ & 344282400 &  74844000& 60  &  16057440       \\ \hline
  $M_1$ & 344282400 &   224532000& 180  &  146149920        \\ \hline
  $M_2$ & 54648000   &  21384000 &   108  &   8320320          \\ \hline
  $M_2$ & 54648000  &    33264000, &   168   &   20200320         \\ \hline
  $M_3$ & 2608200   &    113400&   12  & 4536       \\ \hline
  $M_3$ & 2608200  &     1247400& 132   &  594216          \\ \hline
  $M_4$ & 655776    &    299376 &   126 & 136080          \\ \hline
  $M_4$ &  655776  &   356400  & 150    & 193104            \\ \hline
  $M_5$ & 17931375    &  779625  &  12 &    31185       \\ \hline
  $M_5$ & 17931375   &    4677750&   72 & 1207710             \\ \hline
  $M_5$ & 17931375   &    12474000&   192 &   8663760             \\ \hline
  $M_6$ & 2049300   &  22275 &  3 &    162        \\ \hline
  $M_6$ & 2049300   &  79625 &  105 &  294840            \\ \hline
  $M_6$ & 2049300   &   1247400 &  168 & 757512           \\ \hline
\end{tabular}
\end{subtable}%
\begin{subtable}{0.60\textwidth}
\centering
\fontsize{6pt}{8pt}\selectfont
\begin{tabular}[H]{|c|c|c|c|c|c|} \hline
  Max & $b$ & $r$ & $k$ & $\lambda$ \\ \hline\hline
  $M_7$ & 48600   &  4050 &  23 &   324            \\ \hline
  $M_7$ & 48600  &  44550  &  253 &   40824          \\ \hline
  $M_8$ & 1536975   &    44550 & 8 &  1134           \\ \hline
  $M_8$ & 1536975   &   712800 & 128 & 329184        \\ \hline
  $M_8$ & 1536975   &    779625 & 140 &    394065     \\ \hline
  $M_9$ & 708400    & 15400   &  6 &   280      \\ \hline
  $M_9$ & 708400   &  693000   &  270 &    677880       \\ \hline
  $M_{10}$ & 128800    & 15400   &  33 &  1792       \\ \hline
  $M_{10}$ &  128800  & 113400    &  243 &   99792        \\ \hline
  $M_{11}$ & 170775    &  22275  &  36 & 2835        \\ \hline
  $M_{11}$ &  170775  &  148500   &  240 & 129060          \\ \hline
  $M_{12}$ &  11178    &  4050  & 100 &  1458       \\ \hline
  $M_{12}$ &  11178   & 7128    &  176 & 4536          \\ \hline
  $M_{13}$ & 37950   &  15400   & 112&  6216          \\ \hline
  $M_{13}$ &   37950  &  22275   & 162 &  13041         \\ \hline
\end{tabular}
\end{subtable}
\caption{\small Point and block primitive $2$-designs over $Co_3$ with $v = 276$.}
\end{table}




\end{document}